\documentclass{article}
\usepackage{amsfonts}
\usepackage{amsmath}
\usepackage{shuffle}

\newtheorem{theorem}{Theorem}

\newtheorem{definition}[theorem]{Definition}

\newtheorem{lemma}[theorem]{Lemma}
\newtheorem{notation}[theorem]{Notation}

\newtheorem{proposition}[theorem]{Proposition}

\newenvironment{proof}[1][Proof]{\noindent\textbf{#1.} }{\ \rule{0.5em}{0.5em}}
\begin{document}

\title{Integration of branched rough paths}
\author{Xinru Liu, Danyu Yang}
\maketitle

\begin{abstract}
When the one-form is $Lip\left( \gamma -1\right) $ with $\gamma >p\geq 1$,
we construct the integral of a branched $p$-rough path, which defines
another branched $p$-rough path. We derive a quantitative bound for this
integral and prove that it depends continuously on the driving branched
rough path in rough path metric. Moreover, we prove that the first level
branched rough integral coincides with a first level integral of the
associated $\Pi $-rough path.
\end{abstract}

\section{Introduction}

In his pioneering work \cite{lyons1998differential}, Lyons developed the
theory of rough paths. The theory provides a rigorous framework for solving
differential equations driven by highly oscillatory signals, extending well
beyond the classical setting of Brownian motion and semimartingales. Rough
path theory has found applications in differential equations driven by
non-Markovian processes e.g. Gaussian processes \cite%
{inahama2024moderate,boedihardjo2024lack}, stochastic partial differential
equations \cite{li2025averaging}, robust financial modelling \cite%
{bank2025rough}, numerical analysis \cite{duc2023numerical} and deep
learning \cite{walker2024log,walker2025structured} etc.

The central idea of rough path theory is to enhance an oscillatory path with
higher level information so that the enhanced path takes values in a group.
In the case of geometric rough paths, the group is the free nilpotent Lie
group \cite{lyons1998differential}, which can equivalently be characterised
as a truncated group of characters of the shuffle Hopf algebra \cite%
{reutenauer1993free}. The shuffle product encodes an abstract integration by
parts formula. However, this assumption can be restrictive when dealing with
It\^{o} calculus or with the non-geometric algebraic structures arising in
stochastic partial differential equations.

In \cite{gubinelli2010ramification}, Gubinelli introduced branched rough
paths. Unlike geometric rough paths, branched rough paths take values in a
truncated group of characters of the Connes Kreimer Hopf algebra \cite%
{connes1998hopf}. As a result, branched rough paths naturally accommodate It%
\^{o} calculus, and their associated algebraic structure has been widely
used in solving and analysing stochastic partial differential equations \cite%
{hairer2014theory}.

Lyons \cite{lyons1998differential} defined the solution to a rough
differential equation as a fixed point of the associated full integral
mapping. In contrast, Gubinelli \cite%
{gubinelli2004controlling,gubinelli2010ramification} formulated the solution
as a fixed point in the space of controlled paths. A controlled path takes
values in a vector space and \textquotedblleft remembers\textquotedblright\
the driving rough path as it evolves. This vector-space formulation offers
substantial convenience and skips the explicit construction of the full
integral. However, the construction of the full integral in Lyons' framework
remains important, for instance when studying the differentiability of the
solution path with respect to the initial condition and the driving rough
path.

In \cite{lyons2015theory}, Lyons and Yang introduced the one-form approach
to solving rough differential equations, both in the geometric and branched
settings. In particular, they proved that the differences between successive
Picard iterations decay factorially, and that the solution exists uniquely
and is continuous with respect to the driving rough path. In \cite%
{yang2019algebraic}, Yang lifted polynomial one-forms on the underlying
vector space to exact one-forms on the free nilpotent Lie group, and
interpreted the full geometric rough integral as the integral of a
slowly-varying exact one-form.

In \cite{gyurko2016differential}, Gyurk\'{o} introduced $\Pi $-rough paths
and proved the unique solvability and stability of the solution of a
differential equation driven by a $\Pi $-rough path. $\Pi $-rough paths are
inhomogeneous geometric rough paths, an idea that can be traced back to
Lyons \cite{lyons1998differential}. Based on results in combinatorial
algebra that the Grossman\ Larson algebra is free \cite%
{foissy2002finite,chapoton2010free}, Boedihardjo and Chevyrev \cite%
{boedihardjo2019isomorphism} constructed an isomorphism between branched
rough paths and a class of $\Pi $-rough paths.

In this paper, we work with the vector space of effects associated with a
given branched $p$-rough path for $p\geq 1$. Effects may be compared to
controlled paths; their relationship is analogous to that between indefinite
integrals and integrands. In particular, the series of Picard iterations and
the solution of a differential equation are effects \cite{lyons2015theory}.
Effects systematically compare linear maps at different points via a
parallel translation on the truncated Butcher group, thereby avoiding
complex combinatorics. We prove that effects are stable and continuous in
the space of one-forms under both multiplication (Lemma \ref{Lemma effects
are stable under multiplication}) and iterated integration (Lemma \ref{Lemma
effects are stable under iterated integral}), as required for constructing
the integral.

When the one-form is $Lip\left( \gamma -1\right) $ for $\gamma >p\geq 1$, we
define the integral of a branched $p$-rough path, which is another branched $%
p$-rough path. In Theorem \ref{Theorem integral of branched rough paths}, we
obtain a quantitative bound for this integral and prove that the integral is
continuous with respect to the driving branched rough path in rough path
metric. In Proposition \ref{Proposition same first level integrals}, we
prove that the first level of this integral coincides with a first level
integral of the corresponding $\Pi $-rough path.

\section{Notations and Definitions}

A tree is a finite connected graph without cycles. A rooted tree is a tree
with a distinguished vertex, called the root. We call a rooted tree a tree.
We assume trees are non-planar, so that the children trees of each vertex
are unordered and hence commute. The leaves of a tree are the vertices that
have no children.

A forest is a commutative monomial of trees, including the empty forest
denoted by $\epsilon $. A forest is labelled if each vertex in the forest is
assigned a label from a given label set. We refer to a labelled forest
simply as a forest. The degree $\left\vert \rho \right\vert $ of a forest $%
\rho $ is given by the total number of vertices in $\rho $ with $\left\vert
\epsilon \right\vert =0$. The depth of a forest is the largest number of
vertices on any simple path from a root to a leaf; the depth of $\epsilon $
is $0$. For integer $d\geq 1$, let $T_{d}$($F_{d}$) denote the set of trees
(forests) whose vertices are labelled by $\left\{ 1,\dots ,d\right\} $. For
integer $n\geq 1$, let $T_{d}^{n}$($F_{d}^{n}$) denote the subset of $T_{d}$(%
$F_{d}$) consisting of elements of degrees $1,2,\dots ,n$.

For a labelled forest $\rho $ and $i\in \left\{ 1,\dots ,d\right\} $, let $%
\left[ \rho \right] _{i}$ denote the labelled tree obtained by grafting the
roots of all trees in $\rho $ to a new root with label $i$. Let $\bullet
_{i} $ denote the tree consisting of one vertex labelled by $i$. Following
\cite{gubinelli2010ramification}, define a symmetry factor $\sigma
:F_{d}\rightarrow
\mathbb{N}
$ as $\sigma \left( \bullet _{i}\right) =1$ for $i=1,\dots ,d$ and%
\begin{equation}
\sigma \left( \tau _{1}^{n_{1}}\cdots \tau _{l}^{n_{l}}\right) =\sigma
\left( \left[ \tau _{1}^{n_{1}}\cdots \tau _{l}^{n_{l}}\right] _{i}\right)
=n_{1}!\cdots n_{l}!\sigma \left( \tau _{1}\right) ^{n_{1}}\cdots \sigma
\left( \tau _{l}\right) ^{n_{l}}  \label{definition of symmetry factor}
\end{equation}%
where $i\in \left\{ 1,\dots ,d\right\} $ and $\tau _{j}\in T_{d}$, $%
j=1,2\dots ,l$, are different labelled trees with labels taken into account.

Consider a labelled version of Connes Kreimer Hopf algebra \cite[p.214]%
{connes1998hopf}. The product is given by the commutative multiplication of
monomials of trees. The coproduct is given by admissible cuts \cite[p.215]%
{connes1998hopf}, denoted as $\bigtriangleup $. For integers $n\geq 1$ and $%
d\geq 1$, let $G_{d}^{n}$ denote the step-$n$ truncated group of characters
of Connes Kreimer Hopf algebra labelled with $\left\{ 1,2,\dots ,d\right\} $%
. More specifically, $a\in G_{d}^{n}$ iff $a:F_{d}\rightarrow
\mathbb{R}
$ satisfies%
\begin{equation*}
\left( a,\rho _{1}\rho _{2}\right) =\left( a,\rho _{1}\right) \left( a,\rho
_{2}\right)
\end{equation*}%
for every $\rho _{1},\rho _{2}\in F_{d},\left\vert \rho _{1}\right\vert
+\left\vert \rho _{2}\right\vert \leq n$ and $\left( a,\rho \right) =0$ for $%
\rho \in F_{d},\left\vert \rho \right\vert >n$. The product in $G_{d}^{n}$
is given by, for $a,b\in G_{d}^{n}$, $\rho \in F_{d}$, $\left\vert \rho
\right\vert \leq n$,%
\begin{equation*}
\left( ab,\rho \right) =\sum_{\bigtriangleup \rho }\left( a,\rho _{\left(
1\right) }\right) \left( b,\rho _{\left( 2\right) }\right)
\end{equation*}%
We equip $a\in G_{d}^{n}$ with the norm:%
\begin{equation*}
\left\Vert a\right\Vert :=\max_{\rho \in F_{d},\left\vert \rho \right\vert
=1,\dots ,n}\left\vert \left( a,\rho \right) \right\vert ^{\frac{1}{%
\left\vert \rho \right\vert }}
\end{equation*}

\begin{notation}
A finite partition $D$ of $\left[ 0,T\right] $ consists of points $\left\{
t_{k}\right\} _{k=0}^{n},0=t_{0}<t_{1}<\cdots <t_{n}=T$ for some integer $%
n\geq 1$. Define the mesh of $D$ as $\left\vert D\right\vert
:=\max_{k=0}^{n-1}\left\vert t_{k+1}-t_{k}\right\vert $.
\end{notation}

For $p\geq 1$, let $\left[ p\right] $ denote the largest integer that is
less or equal to $p$.

\begin{definition}[$p$-variation]
Suppose $\left( G,\left\Vert \cdot \right\Vert \right) $ is a topological
group. Let $X:\left[ 0,T\right] \rightarrow G$ be continuous. For $s<t$ in $%
\left[ 0,T\right] $, denote $X_{s,t}:=X_{s}^{-1}X_{t}$. For $p\geq 1$,
define the $p$-variation of $X$ as%
\begin{equation*}
\left\Vert X\right\Vert _{p-var,\left[ 0,T\right] }:=\sup_{D\subset \left[
0,T\right] }\left( \sum_{k=0}^{n-1}\left\Vert X_{t_{k},t_{k+1}}\right\Vert
^{p}\right) ^{\frac{1}{p}}
\end{equation*}%
where the supremum is taken over all finite partitions of $\left[ 0,T\right]
$.
\end{definition}

\begin{definition}[branched $p$-rough path]
For $p\geq 1$, $X:\left[ 0,T\right] \rightarrow G_{d}^{\left[ p\right] }$ is
a branched $p$-rough path if $X$ is continuous and $\left\Vert X\right\Vert
_{p-var,\left[ 0,T\right] }<\infty $.
\end{definition}

Based on \cite%
{foissy2002algebresb,hoffman2003combinatorics,panaite2000relating}, the
graded dual of Connes Kreimer Hopf algebra is isomorphic to Grossman Larson
Hopf algebra \cite{grossman1989hopf} with product \cite[$\left( 3.1\right) $]%
{grossman1989hopf} and coproduct \cite[p.199]{grossman1989hopf}. We assume
Grossman Larson Hopf algebra is a Hopf algebra of forests by deleting the
additional root in \cite{grossman1989hopf}. Based on \cite%
{foissy2002finite,chapoton2010free}, Grossman Larson algebra is freely
generated by a family of trees. Denote this family of trees as $B$ and let $%
B_{d}$ denote the $\left\{ 1,\dots ,d\right\} $-labelled version of $B$.

\begin{notation}
\label{Notation Bd[p]}For $p\geq 1$, let $B_{d}^{\left[ p\right] }=\left\{
\nu _{1},\nu _{2},\cdots ,\nu _{K}\right\} $ denote the set of elements of $%
B_{d}$ whose degree is less or equal to $\left[ p\right] $.
\end{notation}

Then we consider a class of $\Pi $-rough paths that are isomorphic to
branched $p$-rough paths \cite{boedihardjo2019isomorphism}. The framework
for $\Pi $-rough paths in \cite{gyurko2016differential} is more general than
what we present here.

\begin{notation}
Let $W$ denote the set of finite sequences of elements in $\left\{ 1,\dots
,K\right\} $ including the empty sequence denoted as $\eta $. Set $%
\left\vert \eta \right\vert =0$. For $w=i_{1}\cdots i_{m}\in W$, define%
\begin{equation*}
\left\vert w\right\vert :=\left\vert \nu _{i_{1}}\right\vert +\left\vert \nu
_{i_{2}}\right\vert +\cdots +\left\vert \nu _{i_{m}}\right\vert
\end{equation*}%
where $\nu _{i_{j}},j=1,\dots ,m$ are defined in Notation \ref{Notation
Bd[p]}.
\end{notation}

\begin{notation}
Let $G_{\Pi }$ denote the set of mappings $g:\left\{ w\in W|\left\vert
w\right\vert \leq \left[ p\right] \right\} \rightarrow
\mathbb{R}
$ that satisfy%
\begin{equation*}
\left( g,w_{1}\right) \left( g,w_{2}\right) =\left( g,w_{1}\shuffle %
w_{2}\right)
\end{equation*}%
for every $w_{1},w_{2}\in W$, $\left\vert w_{1}\right\vert +\left\vert
w_{2}\right\vert \leq \left[ p\right] $, where $\shuffle$ denotes the
shuffle product \cite[Section 1.4]{reutenauer1993free}.
\end{notation}

For $g,h\in G_{\Pi }$ and $w\in W$, $\left\vert w\right\vert \leq \left[ p%
\right] $, define $\left( gh,w\right) =\sum_{uv=w}\left( g,u\right) \left(
h,v\right) $ with $uv$ denoting the concatenation of $u$ with $v$. Then $%
G_{\Pi }$ is a truncated group of characters of the shuffle Hopf algebra
\cite[p.31]{reutenauer1993free}. We equip $G_{\Pi }$ with the norm:%
\begin{equation*}
\left\Vert g\right\Vert :=\max_{w\in W,\left\vert w\right\vert =1,\dots ,
\left[ p\right] }\left\vert \left( g,w\right) \right\vert ^{\frac{1}{%
\left\vert w\right\vert }}
\end{equation*}

\begin{definition}[$\Pi $-rough path]
Suppose $Z:\left[ 0,T\right] \rightarrow G_{\Pi }$ is continuous. Then $Z$
is a $\Pi $-rough path if $\left\Vert Z\right\Vert _{p-var,\left[ 0,T\right]
}<\infty $.
\end{definition}

For $\gamma >0$, let $\lfloor \gamma \rfloor $ denote the largest integer
that is strictly less than $\gamma $. For vector spaces $U$ and $V$, let $%
L\left( U,V\right) $ denote the set of linear maps from $U$ to $V$. Recall
high order Lipschitz functions defined in \cite[Definition 1.2.4]%
{lyons1998differential}.

\begin{definition}
For $\gamma >0$, $f:%
\mathbb{R}
^{d}\rightarrow L\left(
\mathbb{R}
^{d},%
\mathbb{R}
^{e}\right) $ is $Lip\left( \gamma \right) $ if%
\begin{equation*}
\left\Vert f\right\Vert _{Lip\left( \gamma \right) }:=\max_{k=0,\dots
,\lfloor \gamma \rfloor }\left\Vert D^{k}f\right\Vert _{\infty
}+\sup_{x,y\in
\mathbb{R}
^{d}}\frac{\left\Vert \left( D^{\lfloor \gamma \rfloor }f\right) \left(
x\right) -\left( D^{\lfloor \gamma \rfloor }f\right) \left( y\right)
\right\Vert }{\left\vert x-y\right\vert ^{\gamma -\lfloor \gamma \rfloor }}%
<\infty
\end{equation*}%
where $D^{0}f:=f$, $D^{k}f$ denotes the $k$th Fr\'{e}chet derivative of $f$
for $k=1,\dots ,\lfloor \gamma \rfloor $, and $\left\Vert \cdot \right\Vert
_{\infty }$ denotes the uniform norm.
\end{definition}

Then we define slowly-varying one-form \cite[Definition 8]{yang2019algebraic}
and effect \cite[Definition 9]{yang2019algebraic} in the branched setting.

\begin{notation}
Let $E^{%
\mathbb{R}
^{e}}$ denote the vector bundle on $G_{d}^{\left[ p\right] }$ that
associates each $a\in G_{d}^{\left[ p\right] }$ with the vector space%
\begin{equation*}
E_{a}^{%
\mathbb{R}
^{e}}:=\left\{ \left. \phi :G_{d}^{\left[ p\right] }\rightarrow
\mathbb{R}
^{e}\right\vert \phi =\sum_{\rho \in F_{d}^{\left[ p\right] }}\phi ^{\rho },%
\text{ }\phi ^{\rho }\in L\left(
\mathbb{R}
,%
\mathbb{R}
^{e}\right) \right\}
\end{equation*}%
where $\phi \left( a\right) =\sum_{\rho \in F_{d}^{\left[ p\right] }}\phi
^{\rho }\left( a,\rho \right) $ for $a\in G_{d}^{\left[ p\right] }$. The
parallel translation on $E^{%
\mathbb{R}
^{e}}$ is given by, for $\phi \in E_{a}^{%
\mathbb{R}
^{e}}$ and $b\in G_{d}^{\left[ p\right] }$, define $\phi _{b}\in E_{ab}^{%
\mathbb{R}
^{e}}$ as%
\begin{equation*}
\phi _{b}\left( c\right) :=\phi \left( bc\right) -\phi \left( b\right) \
\text{for }c\in G_{d}^{\left[ p\right] }\text{.}
\end{equation*}
\end{notation}

Denote $\left\Vert \phi \right\Vert :=\max_{\rho \in F_{d}^{\left[ p\right]
}}\left\Vert \phi \right\Vert _{\rho }$ with $\left\Vert \phi \right\Vert
_{\rho }:=\left\vert \phi ^{\rho }\left( 1\right) \right\vert $.

\begin{definition}[slowly-varying one-form]
For $p\geq 1$, let $X:\left[ 0,T\right] \rightarrow G_{d}^{\left[ p\right] }$
be a branched $p$-rough path. Then $\beta :X_{t}\mapsto E_{X_{t}}^{%
\mathbb{R}
^{e}}$ for $t\in \left[ 0,T\right] $ is a slowly-varying one-form along $X$
taking values in $%
\mathbb{R}
^{e}$ if there exists $\theta >1$ such that%
\begin{equation*}
\left\Vert \beta \right\Vert _{\theta }:=\sup_{t\in \left[ 0,T\right]
}\left\Vert \beta \left( X_{t}\right) \right\Vert +\max_{\rho \in F_{d}^{%
\left[ p\right] }}\sup_{0\leq s<t\leq T}\frac{\left\Vert \beta \left(
X_{t}\right) -\beta \left( X_{s}\right) _{X_{s,t}}\right\Vert _{\rho }}{%
\omega \left( s,t\right) ^{\theta -\frac{\left\vert \rho \right\vert }{p}}}%
<\infty
\end{equation*}%
where $\omega \left( s,t\right) :=\left\Vert X\right\Vert _{p-var,\left[ s,t%
\right] }^{p}$.
\end{definition}

\begin{definition}[effect]
For $p\geq 1$, let $X:\left[ 0,T\right] \rightarrow G_{d}^{\left[ p\right] }$
be a branched $p$-rough path. Then $y:\left[ 0,T\right] \rightarrow
\mathbb{R}
^{e}$ is an effect of $X$, if there exist $\xi \in
\mathbb{R}
^{e}$ and a slowly-varying one-form $\beta $ along $X$ taking values in $%
\mathbb{R}
^{e}$ such that%
\begin{equation*}
y_{t}=\xi +\int_{0}^{t}\beta \left( X_{r}\right) dX_{r}\text{ for }t\in %
\left[ 0,T\right] .
\end{equation*}
\end{definition}

The integral is defined as%
\begin{equation*}
\int_{0}^{t}\beta \left( X_{r}\right) dX_{r}:=\lim_{\left\vert D\right\vert
\rightarrow 0}\sum_{k=0}^{n-1}\beta \left( X_{t_{k}}\right) \left(
X_{t_{k},t_{k+1}}\right)
\end{equation*}%
where the limit is taken over all finite partitions of $\left[ 0,t\right] $
as their mesh tends to zero. The limit exists based on the slowly-varying
property of $\beta $ and an argument similar to Young \cite[p.254, result 5]%
{young1936inequality}.

\begin{lemma}
\label{Lemma estimate of increments of effects}For $p\geq 1$, suppose $X:%
\left[ 0,T\right] \rightarrow G_{d}^{\left[ p\right] }$ is a branched $p$%
-rough path. Let $\beta $ be a slowly-varying one-form along $X$ taking
values in $%
\mathbb{R}
^{e}$ such that $\left\Vert \beta \right\Vert _{\theta }<\infty $ for some $%
\theta >1$. Define $y_{t}=\int_{0}^{t}\beta \left( X_{r}\right) dX_{r}$ for $%
t\in \left[ 0,T\right] $. Denote $\omega \left( s,t\right) :=\left\Vert
X\right\Vert _{p-var,\left[ s,t\right] }^{p}$. Then there exists a constant $%
C$ that only depends on $p,d,\theta ,\omega \left( 0,T\right) $ such that
for every $s<t$ in $\left[ 0,T\right] $,%
\begin{eqnarray*}
\left\vert y_{t}-y_{s}-\beta \left( X_{s}\right) \left( X_{s,t}\right)
\right\vert &\leq &C_{p,d,\theta ,\omega \left( 0,T\right) }\left\Vert \beta
\right\Vert _{\theta }\omega \left( s,t\right) ^{\theta } \\
\left\vert y_{t}-y_{s}\right\vert &\leq &C_{p,d,\theta ,\omega \left(
0,T\right) }\left\Vert \beta \right\Vert _{\theta }\omega \left( s,t\right)
^{\frac{1}{p}}.
\end{eqnarray*}
\end{lemma}

\begin{proof}
The first estimate can be proved by sequentially removing a chosen point
from a finite partition as in Young \cite[p.254, result 5]%
{young1936inequality}. The second follows from the first.
\end{proof}

\section{Main Result}

\begin{notation}
For $\gamma >p\geq 1$, suppose $X:\left[ 0,T\right] \rightarrow G_{d}^{\left[
p\right] }$ is a branched $p$-rough path and $f=\left( f_{1},\cdots
,f_{d}\right) :%
\mathbb{R}
^{d}\rightarrow L\left(
\mathbb{R}
^{d},%
\mathbb{R}
^{e}\right) $ is a $Lip\left( \gamma -1\right) $ one-form where $f_{i}:%
\mathbb{R}
^{d}\rightarrow
\mathbb{R}
^{e}$, $i=1,2,\dots ,d$.
\end{notation}

\begin{notation}
\label{Notation f tau in integrals}If $\tau =\bullet _{i}$ for $i\in \left\{
1,\dots ,d\right\} $, define $f_{\tau }:=f_{i}$; if $\tau =\left[ \bullet
_{i_{1}}\cdots \bullet _{i_{l}}\right] _{i}$ for $i_{j},i\in \left\{ 1,\dots
,d\right\} $ and $l\in \left\{ 1,\dots ,\left[ p\right] -1\right\} $, define
$f_{\tau }:=D_{i_{1},\cdots ,i_{l}}^{l}f_{i}$; if the depth of $\tau \in
T_{d}^{\left[ p\right] }$ is greater than $2$, define $f_{\tau }:=0$.
\end{notation}

\begin{lemma}
\label{Lemma beta is a slowly varying one-form}With the symmetry factor $%
\sigma $ defined at $\left( \ref{definition of symmetry factor}\right) $,
define%
\begin{equation}
\beta \left( X_{t}\right) \left( a\right) :=\sum_{\tau \in T_{d},\left\vert
\tau \right\vert =1,\dots ,\left[ p\right] }f_{\tau }\left( x_{t}\right)
\frac{\left( a,\tau \right) }{\sigma \left( \tau \right) }\text{ for }a\in
G_{d}^{\left[ p\right] },  \label{Definition beta first level one-form}
\end{equation}%
where $f_{\tau }$ is defined in Notation \ref{Notation f tau in integrals}
and $x_{t}$ denotes the projection of $X_{t}$ to $%
\mathbb{R}
^{d}$. Then $\beta $ is a slowly-varying one-form along $X$ taking values in
$%
\mathbb{R}
^{e}$, and with $\theta :=\frac{\gamma \wedge \left( \left[ p\right]
+1\right) }{p}>1$,
\begin{equation*}
\left\Vert \beta \right\Vert _{\theta }\leq \left\Vert f\right\Vert
_{Lip\left( \gamma -1\right) }.
\end{equation*}
\end{lemma}

$\beta \left( X_{t}\right) \left( a\right) $ at $\left( \ref{Definition beta
first level one-form}\right) $ only involves the tree components of $a$,
which is important for the definition of iterated integral in $\left( \ref%
{definition of Y bracket sigma}\right) $ below.

\begin{definition}
\label{Definition of rough integral}With $\beta $ defined at $\left( \ref%
{Definition beta first level one-form}\right) $, denote $y_{t}=\left(
y_{t}^{1},\cdots ,y_{t}^{e}\right) :=\int_{0}^{t}\beta \left( X_{r}\right)
dX_{r}$ for $t\in \left[ 0,T\right] $. Define $Y:\left\{ \left( s,t\right)
|0\leq s<t\leq T\right\} \rightarrow G_{e}^{\left[ p\right] }$ inductively
as $\left( Y_{s,t},\bullet _{j}\right) :=y_{t}^{j}-y_{s}^{j}$ for $j=1,\dots
,e$; for a forest $\rho \in F_{e}$, $\left\vert \rho \right\vert =1,\dots ,%
\left[ p\right] -1$ and $j=1,\dots ,e$,
\begin{equation}
\left( Y_{s,t},\left[ \rho \right] _{j}\right) :=\int_{r=s}^{t}\left(
Y_{s,r},\rho \right) dy_{r}^{j}\text{;}
\label{definition of Y bracket sigma}
\end{equation}%
for trees $\tau _{j}\in T_{e}^{\left[ p\right] }$, $j=1,\dots ,l$, $%
\left\vert \tau _{1}\right\vert +\cdots +\left\vert \tau _{l}\right\vert
\leq \left[ p\right] $,
\begin{equation*}
\left( Y_{s,t},\tau _{1}\cdots \tau _{l}\right) :=\left( Y_{s,t},\tau
_{1}\right) \cdots \left( Y_{s,t},\tau _{l}\right) .
\end{equation*}%
We call $Y$ the rough integral of $f$ along $X$ and denote $Y_{s,t}$ as $%
\int_{s}^{t}f\left( X_{r}\right) dX_{r}$.
\end{definition}

The existence of $Y$ follows from the stability of effects under
multiplication (Lemma \ref{Lemma effects are stable under multiplication})
and iterated integration (Lemma \ref{Lemma effects are stable under iterated
integral}). That $Y_{s,u}Y_{u,t}=Y_{s,t}$ for $s<u<t$ in $\left[ 0,T\right] $
can be proved by induction on the degree of forests, using the equality
\begin{equation*}
\bigtriangleup \left[ \rho \right] _{j}=\left[ \rho \right] _{j}\otimes
1+\sum_{\bigtriangleup \rho }\rho _{\left( 1\right) }\otimes \left[ \rho
_{\left( 2\right) }\right] _{j}
\end{equation*}%
for a forest $\rho \in F_{e},\left\vert \rho \right\vert =0,\dots ,\left[ p%
\right] -1$ and $j\in \left\{ 1,\dots ,e\right\} $.

Let $\pi _{1}:G_{d}^{\left[ p\right] }\rightarrow
\mathbb{R}
^{d}\ $denote the projection to degree-$1$ components. Define the $d_{p}$%
-metric between branched $p$-rough paths $X^{i}:\left[ 0,T\right]
\rightarrow G_{d}^{\left[ p\right] },i=1,2$ as%
\begin{eqnarray*}
d_{p}\left( X^{1},X^{2}\right) &:=&\left\vert \pi _{1}\left(
X_{0}^{1}\right) -\pi _{1}\left( X_{0}^{2}\right) \right\vert \\
&&+\max_{\rho \in F_{d}^{\left[ p\right] }}\sup_{D\subset \left[ 0,T\right]
}\left( \sum_{k=0}^{n-1}\left\vert \left( X_{t_{k},t_{k+1}}^{1},\rho \right)
-\left( X_{t_{k},t_{k+1}}^{2},\rho \right) \right\vert ^{\frac{p}{\left\vert
\rho \right\vert }}\right) ^{\frac{1}{p}}
\end{eqnarray*}%
where the supremum is taken over all finite partitions of $\left[ 0,T\right]
$.

\begin{theorem}
\label{Theorem integral of branched rough paths}The rough integral $%
\int_{0}^{\cdot }f\left( X_{r}\right) dX_{r}:\left[ 0,T\right] \rightarrow
G_{e}^{\left[ p\right] }$ in Definition \ref{Definition of rough integral}
exists and is continuous with respect to $X$ in $d_{p}$-metric. Moreover,
there exists a constant $C$ depending only on $p,\gamma ,d,e$ such that for
all $s<t$ in $\left[ 0,T\right] $,%
\begin{equation*}
\left\Vert \int f\left( X_{r}\right) dX_{r}\right\Vert _{p-var,\left[ s,t%
\right] }\leq C\left\Vert f\right\Vert _{Lip\left( \gamma -1\right) }\left(
\left\Vert X\right\Vert _{p-var,\left[ s,t\right] }\vee \left\Vert
X\right\Vert _{p-var,\left[ s,t\right] }^{p}\right) .
\end{equation*}
\end{theorem}

Recall Notation \ref{Notation Bd[p]} that $\left\{ \nu _{1},\nu _{2},\cdots
,\nu _{K}\right\} $ is the set of labelled trees that freely generate the
step-$\left[ p\right] $ truncated labelled Grossman Larson algebra. Based on
\cite[Theorem 4.3]{boedihardjo2019isomorphism}, there exists a $\Pi $-rough
path $Z$ that is isomorphic to branched rough path $X$. The integral of $\Pi
$-rough path is defined in \cite[Definition 3.10]{gyurko2016differential}.

\begin{proposition}
\label{Proposition same first level integrals}Suppose $Z$ is the $\Pi $%
-rough path that is isomorphic to $X$. Define $g=\left( g_{1},\cdots
,g_{K}\right) :%
\mathbb{R}
^{d}\rightarrow L\left(
\mathbb{R}
^{K},%
\mathbb{R}
^{e}\right) $ as $g_{k}:=f_{\nu _{k}},k=1,\dots ,K$ with $f_{\nu _{k}}$
defined in Notation \ref{Notation f tau in integrals}. Then for any $s<t$ in
$\left[ 0,T\right] $ the first level integral of $f$ along $X$ on $\left[ s,t%
\right] $ is equal to the first level integral of $g$ along $Z$ on $\left[
s,t\right] $.
\end{proposition}

\section{Proofs}

For $s<t$ in $\left[ 0,T\right] $, denote $\omega \left( s,t\right)
:=\left\Vert X\right\Vert _{p-var,\left[ s,t\right] }^{p}$. Assume $\gamma
\in \left( p,\left[ p\right] +1\right] $.

\begin{proof}[Proof of Lemma \protect\ref{Lemma beta is a slowly varying
one-form}]
For $s<t$ in $\left[ 0,T\right] $ and $a\in G_{d}^{\left[ p\right] }$,%
\begin{eqnarray*}
&&\left( \beta \left( X_{t}\right) -\beta \left( X_{s}\right)
_{X_{s,t}}\right) \left( a\right) \\
&=&\beta \left( X_{t}\right) \left( a\right) -\beta \left( X_{s}\right)
\left( X_{s,t}a\right) +\beta \left( X_{s}\right) \left( X_{s,t}\right) \\
&=&\sum_{\tau \in T_{d}^{\left[ p\right] }}\frac{1}{\sigma \left( \tau
\right) }\left( f_{\tau }\left( x_{t}\right) -\sum_{\substack{ i_{j}\in
\left\{ 1,\dots ,d\right\}  \\ l\in \left\{ 0,\dots ,\left[ p\right]
-\left\vert \tau \right\vert \right\} }}\frac{D_{i_{1},\cdots
,i_{l}}^{l}f_{\tau }\left( x_{s}\right) }{l!}\left( X_{s,t},\bullet
_{i_{1}}\cdots \bullet _{i_{l}}\right) \right) \left( a,\tau \right) \\
&&+\sum_{\tau \in T_{d}^{\left[ p\right] }}\frac{1}{\sigma \left( \tau
\right) }\sum_{\substack{ i_{j}\in \left\{ 1,\dots ,d\right\}  \\ l\in
\left\{ 0,\dots ,\left[ p\right] -\left\vert \tau \right\vert \right\} }}%
\frac{D_{i_{1},\cdots ,i_{l}}^{l}f_{\tau }\left( x_{s}\right) }{l!}\left(
X_{s,t},\bullet _{i_{1}}\cdots \bullet _{i_{l}}\right) \left( a,\tau \right)
\\
&&-\sum_{\tau \in T_{d}^{\left[ p\right] }}\frac{f_{\tau }\left(
x_{s}\right) }{\sigma \left( \tau \right) }\left( \left( X_{s,t}a,\tau
\right) -\left( X_{s,t},\tau \right) \right) \\
&=:&A+B-C
\end{eqnarray*}

Then we prove that $B=C$. Denote the set of $\left\{ 1,\dots ,d\right\} $%
-labelled forests of depth less or equal to $1$ (including the empty forest)
as $S_{d}$. Since the number of ordered sequences $i_{1}\cdots i_{l}$ that
correspond to the same forest $\rho =\bullet _{i_{1}}\cdots \bullet _{i_{l}}$
is $l!/\sigma \left( \rho \right) $ where $\sigma \left( \rho \right) $ is
the symmetry factor of $\rho $ defined at $\left( \ref{definition of
symmetry factor}\right) $. For the empty forest $\epsilon $, denote $%
D^{\epsilon }f_{\tau }:=f_{\tau }$. For $\rho =\bullet _{i_{1}}\cdots
\bullet _{i_{l}}$, denote $D^{\rho }f_{\tau }:=D_{i_{1},\cdots
,i_{l}}^{l}f_{\tau }$, where the derivative is independent of the order of
differentiation. Then%
\begin{equation*}
B=\sum_{\tau \in T_{d}^{\left[ p\right] }}\sum_{\rho \in S_{d},\left\vert
\rho \right\vert \leq \left[ p\right] -\left\vert \tau \right\vert }\frac{%
D^{\rho }f_{\tau }\left( x_{s}\right) }{\sigma \left( \tau \right) \sigma
\left( \rho \right) }\left( X_{s,t},\rho \right) \left( a,\tau \right) .
\end{equation*}%
Let $\ast $ denote the product of Grossman Larson Hopf algebra \cite[$\left(
3.1\right) $]{grossman1989hopf}. Based on \cite[Theorem 43]%
{foissy2002algebresb}\cite[Proposition 4.4]{hoffman2003combinatorics},
\begin{equation*}
\left( X_{s,t}a,\tau \right) =\sum_{\rho _{i}\in F_{d},\left\vert \rho
_{i}\right\vert =0,\dots ,\left\vert \tau \right\vert }\frac{\sigma \left(
\tau \right) }{\sigma \left( \rho _{1}\right) \sigma \left( \rho _{2}\right)
}\left( \rho _{1}\ast \rho _{2},\tau \right) \left( X_{s,t},\rho _{1}\right)
\left( a,\rho _{2}\right) .
\end{equation*}%
Since $f_{\tau }=0$ when the depth of $\tau $ is greater than $2$, we have,
for forests $\rho _{i}$, $i=1,2$ with $\left\vert \rho _{2}\right\vert \geq
1 $,
\begin{equation*}
\sum_{\tau \in T_{d}^{\left[ p\right] }}\left( \rho _{1}\ast \rho _{2},\tau
\right) f_{\tau }=D^{\rho _{1}}f_{\rho _{2}},
\end{equation*}%
which can be non-zero only if $\rho _{2}$ is a tree of depth $1$ or $2$ and $%
\rho _{1}\in S_{d}$. Then%
\begin{eqnarray*}
C &=&\sum_{\tau \in T_{d}^{\left[ p\right] }}\frac{f_{\tau }\left(
x_{s}\right) }{\sigma \left( \tau \right) }\sum_{\rho _{2}\in T_{d}^{\left[ p%
\right] }}\sum_{\rho _{1}\in S_{d}}\frac{\sigma \left( \tau \right) }{\sigma
\left( \rho _{1}\right) \sigma \left( \rho _{2}\right) }\left( \rho _{1}\ast
\rho _{2},\tau \right) \left( X_{s,t},\rho _{1}\right) \left( a,\rho
_{2}\right) \\
&=&\sum_{\rho _{2}\in T_{d}^{\left[ p\right] }}\sum_{\rho _{1}\in
S_{d},\left\vert \rho _{1}\right\vert \leq \left[ p\right] -\left\vert \rho
_{2}\right\vert }\frac{D^{\rho _{1}}f_{\rho _{2}}\left( x_{s}\right) }{%
\sigma \left( \rho _{1}\right) \sigma \left( \rho _{2}\right) }\left(
X_{s,t},\rho _{1}\right) \left( a,\rho _{2}\right) \\
&=&B
\end{eqnarray*}

To estimate $A$, since $f$ is $Lip\left( \gamma -1\right) $, based on the
Taylor expansion of $f$ and that $\sigma \left( \tau \right) \geq 1$, we
have, $\left\Vert \beta \left( X_{t}\right) \right\Vert \leq \left\Vert
f\right\Vert _{Lip\left( \gamma -1\right) }$ and for every $\rho \in F_{d}^{%
\left[ p\right] }$,%
\begin{equation*}
\left\Vert \beta \left( X_{t}\right) -\beta \left( X_{s}\right)
_{X_{s,t}}\right\Vert _{\rho }\leq \left\Vert f\right\Vert _{Lip\left(
\gamma -1\right) }\omega \left( s,t\right) ^{\frac{\gamma -\left\vert \rho
\right\vert }{p}}.
\end{equation*}%
As a result, $\left\Vert \beta \right\Vert _{\frac{\gamma }{p}}\leq
\left\Vert f\right\Vert _{Lip\left( \gamma -1\right) }$.
\end{proof}

\begin{notation}
\label{Notation multiplication of linear maps}For $i=1,\cdots ,l$, assume $%
\phi _{i}:G_{d}^{\left[ p\right] }\rightarrow
\mathbb{R}
$, $\phi _{i}=\sum_{\rho _{i}\in F_{d}^{\left[ p\right] }}\phi _{i}^{\rho
_{i}}$ with $\phi _{i}^{\rho _{i}}\in
\mathbb{R}
$. Define $\phi _{1}\cdots \phi _{l}:G_{d}^{\left[ p\right] }\rightarrow
\mathbb{R}
$ as%
\begin{equation*}
\left( \phi _{1}\cdots \phi _{l}\right) \left( a\right) :=\sum_{\rho _{i}\in
F_{d}^{\left[ p\right] },\left\vert \rho _{1}\right\vert +\cdots +\left\vert
\rho _{l}\right\vert \leq \left[ p\right] }\phi _{1}^{\rho _{1}}\cdots \phi
_{l}^{\rho _{l}}\left( a,\rho _{1}\cdots \rho _{l}\right) \text{ for }a\in
G_{d}^{\left[ p\right] }.
\end{equation*}
\end{notation}

\begin{lemma}
\label{Lemma effects are stable under multiplication}Suppose $\beta
_{i},i=1,2$ are slowly-varying one-forms along $X$ taking values in $%
\mathbb{R}
$. Suppose for some $\theta \in \left( 1,\frac{\left[ p\right] +1}{p}\right]
$, $\left\Vert \beta _{i}\right\Vert _{\theta }<\infty $ for $i=1,2$. For $%
t\in \left[ 0,T\right] $ and $a\in G_{d}^{\left[ p\right] }$, define%
\begin{eqnarray*}
\beta \left( X_{t}\right) \left( a\right) &:=&\int_{0}^{t}\beta _{1}\left(
X_{r}\right) dX_{r}\beta _{2}\left( X_{t}\right) \left( a\right) +\beta
_{1}\left( X_{t}\right) \left( a\right) \int_{0}^{t}\beta _{2}\left(
X_{r}\right) dX_{r} \\
&&+\left( \beta _{1}\left( X_{t}\right) \beta _{2}\left( X_{t}\right)
\right) \left( a\right) .
\end{eqnarray*}%
Then $\beta $ is a slowly-varying one-form along $X$ taking values in $%
\mathbb{R}
$ and there exists a constant $C$ depending only on $p,d,\theta ,\omega
\left( 0,T\right) $ such that%
\begin{equation*}
\left\Vert \beta \right\Vert _{\theta }\leq C_{p,d,\theta ,\omega \left(
0,T\right) }\left\Vert \beta _{1}\right\Vert _{\theta }\left\Vert \beta
_{2}\right\Vert _{\theta }
\end{equation*}
\end{lemma}

\begin{proof}
For $s<t$ in $\left[ 0,T\right] $ and $a\in G_{d}^{\left[ p\right] }$,%
\begin{eqnarray*}
&&\left( \beta \left( X_{t}\right) -\beta \left( X_{s}\right)
_{X_{s,t}}\right) \left( a\right) \\
&=&\int_{0}^{s}\beta _{1}\left( X_{r}\right) dX_{r}\left( \beta _{2}\left(
X_{t}\right) -\left( \beta _{2}\left( X_{s}\right) \right) _{X_{s,t}}\right)
\left( a\right) +\left( \int_{s}^{t}\beta _{1}\left( X_{r}\right)
dX_{r}-\beta _{1}\left( X_{s}\right) \left( X_{s,t}\right) \right) \beta
_{2}\left( X_{t}\right) \left( a\right) \\
&&+\left( \beta _{1}\left( X_{t}\right) -\beta _{1}\left( X_{s}\right)
_{X_{s,t}}\right) \left( a\right) \int_{0}^{s}\beta _{2}\left( X_{r}\right)
dX_{r}+\beta _{1}\left( X_{t}\right) \left( a\right) \left(
\int_{s}^{t}\beta _{2}\left( X_{r}\right) dX_{r}-\beta _{2}\left(
X_{s}\right) \left( X_{s,t}\right) \right) \\
&&+\beta _{1}\left( X_{s}\right) \left( X_{s,t}\right) \left( \beta
_{2}\left( X_{t}\right) -\beta _{2}\left( X_{s}\right) _{X_{s,t}}\right)
\left( a\right) +\left( \beta _{1}\left( X_{t}\right) -\beta _{1}\left(
X_{s}\right) _{X_{s,t}}\right) \left( a\right) \beta _{2}\left( X_{s}\right)
\left( X_{s,t}\right) \\
&&+\left( \beta _{1}\left( X_{t}\right) \beta _{2}\left( X_{t}\right) -\beta
_{1}\left( X_{s}\right) _{X_{s,t}}\beta _{2}\left( X_{s}\right)
_{X_{s,t}}\right) \left( a\right) \\
&&+\sum_{\substack{ \left\vert \rho _{1}\right\vert +\left\vert \rho
_{2}\right\vert \geq \left[ p\right] +1  \\ \rho _{i}\in F_{d}^{\left[ p%
\right] }}}\sum_{\substack{ \bigtriangleup \rho _{1},\bigtriangleup \rho
_{2}  \\ 1\leq \left\vert \left( \rho _{1}\right) _{\left( 2\right) }\left(
\rho _{2}\right) _{\left( 2\right) }\right\vert \leq \left[ p\right] }}\beta
_{1}^{\rho _{1}}\left( X_{s}\right) \beta _{2}^{\rho _{2}}\left(
X_{s}\right) \left( X_{s,t},\left( \rho _{1}\right) _{\left( 1\right)
}\right) \left( X_{s,t},\left( \rho _{2}\right) _{\left( 1\right) }\right)
\left( a,\left( \rho _{1}\right) _{\left( 2\right) }\left( \rho _{2}\right)
_{\left( 2\right) }\right) ,
\end{eqnarray*}%
where $\beta _{i}=\sum_{\rho _{i}\in F_{d}^{\left[ p\right] }}\beta
_{i}^{\rho _{i}}$ and $\bigtriangleup \rho _{i}=\sum_{\bigtriangleup \rho
_{i}}\left( \rho _{i}\right) _{\left( 1\right) }\otimes \left( \rho
_{i}\right) _{\left( 2\right) },i=1,2$. Combine the assumption with Lemma %
\ref{Lemma estimate of increments of effects}, the result follows.
\end{proof}

\begin{notation}
For a forest $\rho \ $and a tree $\tau $, $\left\vert \tau \right\vert \geq
1 $, let $\rho \succ \tau $ denote the tree of degree $\left\vert \rho
\right\vert +\left\vert \tau \right\vert $ obtained by grafting the roots of
all trees in $\rho $ to the root of $\tau $. For the empty forest $\epsilon $%
, $\epsilon \succ \tau :=\tau $ for any tree $\tau $, $\left\vert \tau
\right\vert \geq 1$.
\end{notation}

\begin{notation}
\label{Notation iterated integral of linear maps}For $\phi _{i}:G_{d}^{\left[
p\right] }\rightarrow
\mathbb{R}
,i=1,2$, assume $\phi _{1}=\sum_{\rho \in F_{d}^{\left[ p\right] }}\phi
_{1}^{\rho }$ with $\phi _{1}^{\rho }\in
\mathbb{R}
$ and $\phi _{2}=\sum_{\tau \in T_{d}^{\left[ p\right] }}\phi _{2}^{\tau }$
with $\phi _{2}^{\tau }\in
\mathbb{R}
$. Define $\phi _{1}\succ \phi _{2}:G_{d}^{\left[ p\right] }\rightarrow
\mathbb{R}
$ as
\begin{equation*}
\left( \phi _{1}\succ \phi _{2}\right) \left( a\right) :=\sum_{\rho \in
F_{d}^{\left[ p\right] },\tau \in T_{d}^{\left[ p\right] },\left\vert \rho
\right\vert +\left\vert \tau \right\vert \leq \left[ p\right] }\phi
_{1}^{\rho }\phi _{2}^{\tau }\left( a,\rho \succ \tau \right) \text{ for }%
a\in G_{d}^{\left[ p\right] }.
\end{equation*}
\end{notation}

\begin{lemma}
\label{Lemma effects are stable under iterated integral}Suppose $\beta
_{i},i=1,2$ are two slowly-varying one-forms along $X$ taking values in $%
\mathbb{R}
$. Assume for some $\theta \in \left( 1,\frac{\left[ p\right] +1}{p}\right] $%
, $\left\Vert \beta _{i}\right\Vert _{\theta }<\infty $, $i=1,2$. Further
assume that $\beta _{2}=\sum_{\tau \in T_{d}^{\left[ p\right] }}\beta
_{2}^{\tau }$ only contains tree components. For $t\in \left[ 0,T\right] $
and $a\in G_{d}^{\left[ p\right] }$, define%
\begin{equation*}
\beta \left( X_{t}\right) \left( a\right) :=\int_{0}^{t}\beta _{1}\left(
X_{r}\right) dX_{r}\beta _{2}\left( X_{t}\right) \left( a\right) +\left(
\beta _{1}\left( X_{t}\right) \succ \beta _{2}\left( X_{t}\right) \right)
\left( a\right) .
\end{equation*}%
Then $\beta $ is a slowly-varying one-form along $X$ taking values in $%
\mathbb{R}
$ and there exists a constant $C$ only depending on $p,d,\theta ,\omega
\left( 0,T\right) $ such that%
\begin{equation*}
\left\Vert \beta \right\Vert _{\theta }\leq C_{p,d,\theta ,\omega \left(
0,T\right) }\left\Vert \beta _{1}\right\Vert _{\theta }\left\Vert \beta
_{2}\right\Vert _{\theta }
\end{equation*}
\end{lemma}

\begin{proof}
For a forest $\rho $ and a tree $\tau $, $\left\vert \tau \right\vert \geq 1$%
,%
\begin{equation*}
\bigtriangleup \left( \rho \succ \tau \right) =\left( \rho \succ \tau
\right) \otimes 1+\sum_{\bigtriangleup \rho ,\bigtriangleup \tau ,\left\vert
\tau _{\left( 2\right) }\right\vert \geq 1}\left( \rho _{\left( 1\right)
}\tau _{\left( 1\right) }\right) \otimes \left( \rho _{\left( 2\right)
}\succ \tau _{\left( 2\right) }\right)
\end{equation*}%
where $\bigtriangleup \rho =\sum_{\bigtriangleup \rho }\rho _{\left(
1\right) }\otimes \rho _{\left( 2\right) }$ and $\bigtriangleup \tau
=\sum_{\bigtriangleup \tau }\tau _{\left( 1\right) }\otimes \tau _{\left(
2\right) }$.

For $s<t$ in $\left[ 0,T\right] $ and $a\in G_{d}^{\left[ p\right] }$,%
\begin{eqnarray*}
&&\left( \beta \left( X_{t}\right) -\beta \left( X_{s}\right)
_{X_{s,t}}\right) \left( a\right) \\
&=&\int_{0}^{s}\beta _{1}\left( X_{r}\right) dX_{r}\left( \beta _{2}\left(
X_{t}\right) -\beta _{2}\left( X_{s}\right) _{X_{s,t}}\right) \left( a\right)
\\
&&+\left( \int_{s}^{t}\beta _{1}\left( X_{r}\right) dX_{r}-\beta _{1}\left(
X_{s}\right) \left( X_{s,t}\right) \right) \beta _{2}\left( X_{t}\right)
\left( a\right) \\
&&+\beta _{1}\left( X_{s}\right) \left( X_{s,t}\right) \left( \beta
_{2}\left( X_{t}\right) -\beta _{2}\left( X_{s}\right) _{X_{s,t}}\right)
\left( a\right) \\
&&+\left( \beta _{1}\left( X_{t}\right) \succ \beta _{2}\left( X_{t}\right)
-\beta _{1}\left( X_{s}\right) _{X_{s,t}}\succ \beta _{2}\left( X_{s}\right)
_{X_{s,t}}\right) \left( a\right) \\
&&+\sum_{\substack{ \left\vert \rho \right\vert +\left\vert \tau \right\vert
\geq \left[ p\right] +1  \\ \rho \in F_{d}^{\left[ p\right] },\tau \in
T_{d}^{\left[ p\right] }}}\sum_{\substack{ \bigtriangleup \rho
,\bigtriangleup \tau  \\ \left\vert \rho _{\left( 2\right) }\right\vert
+\left\vert \tau _{\left( 2\right) }\right\vert \leq \left[ p\right]
,\left\vert \tau _{\left( 2\right) }\right\vert \geq 1}}\beta _{1}^{\rho
}\left( X_{s}\right) \beta _{2}^{\tau }\left( X_{s}\right) \left(
X_{s,t},\rho _{\left( 1\right) }\right) \left( X_{s,t},\tau _{\left(
1\right) }\right) \left( a,\rho _{\left( 2\right) }\succ \tau _{\left(
2\right) }\right)
\end{eqnarray*}%
where $\beta _{1}=\sum_{\rho \in F_{d}^{\left[ p\right] }}\beta _{1}^{\rho }$
and $\beta _{2}=\sum_{\tau \in T_{d}^{\left[ p\right] }}\beta _{2}^{\tau }$.
Combine the assumption with Lemma \ref{Lemma estimate of increments of
effects}, the result follows.
\end{proof}

\begin{definition}
\label{Definition of Y tilde}Recall $\beta =\left( \beta ^{1},\cdots ,\beta
^{e}\right) $ defined at $\left( \ref{Definition beta first level one-form}%
\right) $. Fix $s\in \left[ 0,T\right) $. For $\rho \in F_{e}^{\left[ p%
\right] }$, define $B_{\rho }:G_{d}^{\left[ p\right] }\rightarrow
\mathbb{R}
$ recursively as $B_{\bullet _{j}}:=\beta ^{j}\left( X_{s}\right) $, $%
j=1,\dots ,e$; for trees $\tau _{k}\in T_{e}^{\left[ p\right] }$, $k=1,\dots
,l$, $\left\vert \tau _{1}\right\vert +\cdots +\left\vert \tau
_{l}\right\vert \leq \left[ p\right] $, with $B_{\tau _{1}}\cdots B_{\tau
_{l}}$ defined in Notation \ref{Notation multiplication of linear maps},
define%
\begin{equation*}
B_{\tau _{1}\cdots \tau _{l}}:=B_{\tau _{1}}\cdots B_{\tau _{l}}\text{;}
\end{equation*}%
for forest $\rho \in F_{e}^{\left[ p\right] }$, $\left\vert \rho \right\vert
\leq \left[ p\right] -1$ and $j=1,\dots ,e$, with $B_{\rho }\succ B_{\bullet
_{j}}$ defined in Notation \ref{Notation iterated integral of linear maps},
define%
\begin{equation*}
B_{\left[ \rho \right] _{j}}:=B_{\rho }\succ B_{\bullet _{j}}\text{.}
\end{equation*}%
For $t\in (s,T]$, define $\widetilde{Y}_{s,t}:\left\{ \rho \in
F_{e}|\left\vert \rho \right\vert \leq \left[ p\right] \right\} \rightarrow
\mathbb{R}
\ $as $\left( \widetilde{Y}_{s,t},\epsilon \right) :=1$ and for $\rho \in
F_{e}^{\left[ p\right] }$,
\begin{equation*}
\left( \widetilde{Y}_{s,t},\rho \right) :=B_{\rho }\left( X_{s,t}\right) .
\end{equation*}
\end{definition}

\begin{lemma}
Recall $Y\ $in Definition \ref{Definition of rough integral}. Then there
exists a constant $C$ that only depends on $p,\gamma ,d,\omega \left(
0,T\right) $ such that for every $s<t$ in $\left[ 0,T\right] $ and every $%
\rho \in F_{e}^{\left[ p\right] }$,%
\begin{eqnarray}
\left\vert \left( Y_{s,t},\rho \right) -\left( \widetilde{Y}_{s,t},\rho
\right) \right\vert &\leq &C_{p,\gamma ,d,\omega \left( 0,T\right)
}\left\Vert f\right\Vert _{Lip\left( \gamma -1\right) }^{\left\vert \rho
\right\vert }\omega \left( s,t\right) ^{\frac{\gamma }{p}}
\label{estimate Y s,t rho and Y tilde s,t rho are close} \\
\left\vert \left( Y_{s,t},\rho \right) \right\vert ^{\frac{1}{\left\vert
\rho \right\vert }} &\leq &C_{p,\gamma ,d,\omega \left( 0,T\right)
}\left\Vert f\right\Vert _{Lip\left( \gamma -1\right) }\omega \left(
s,t\right) ^{\frac{1}{p}}  \label{estimate norm on Ys,t rho}
\end{eqnarray}
\end{lemma}

\begin{proof}
Based on Definition \ref{Definition of Y tilde}, it can be proved
inductively that, for every $\rho \in F_{e}^{\left[ p\right] }$,
\begin{equation}
\left\vert \left( \widetilde{Y}_{s,t},\rho \right) \right\vert \leq
C_{p,\gamma ,d,\omega \left( 0,T\right) }\left\Vert f\right\Vert _{Lip\left(
\gamma -1\right) }^{\left\vert \rho \right\vert }\omega \left( s,t\right) ^{%
\frac{\left\vert \rho \right\vert }{p}}.
\label{inner estimate Y tilde s,t rho}
\end{equation}%
We only prove $\left( \ref{estimate Y s,t rho and Y tilde s,t rho are close}%
\right) $. Then $\left( \ref{estimate norm on Ys,t rho}\right) $ follows
from $\left( \ref{estimate Y s,t rho and Y tilde s,t rho are close}\right) $
and $\left( \ref{inner estimate Y tilde s,t rho}\right) $.

When $\left\vert \rho \right\vert =1$, $\left( \ref{estimate Y s,t rho and Y
tilde s,t rho are close}\right) $ is true based on Lemma \ref{Lemma beta is
a slowly varying one-form} and Lemma \ref{Lemma estimate of increments of
effects}. Suppose $\left( \ref{estimate Y s,t rho and Y tilde s,t rho are
close}\right) $ holds for $\left\vert \rho \right\vert \leq k,\rho \in
F_{e},k=1,\dots ,\left[ p\right] -1$.

Suppose $\left\vert \rho \right\vert =k+1$ and $\rho =\rho _{1}\rho _{2}$
for $\rho _{i}\in F_{e}^{\left[ p\right] },\left\vert \rho _{i}\right\vert
\geq 1$. Based on Definition \ref{Definition of Y tilde}, when $B_{\rho
}=\sum_{\lambda \in F_{d}^{\left[ p\right] }}B_{\rho }^{\lambda }$, we have $%
\left\vert B_{\rho }^{\lambda }\right\vert \leq C_{p,d}\left\Vert
f\right\Vert _{Lip\left( \gamma -1\right) }^{\left\vert \rho \right\vert }$.
Then based on the inductive hypothesis, $\left( \ref{inner estimate Y tilde
s,t rho}\right) $ and that $\gamma \in (p,\left[ p\right] +1]$,
\begin{eqnarray*}
&&\left\vert \left( Y_{s,t},\rho \right) -\left( \widetilde{Y}_{s,t},\rho
\right) \right\vert \\
&\leq &\left\vert \left( Y_{s,t},\rho _{1}\right) -\left( \widetilde{Y}%
_{s,t},\rho _{1}\right) \right\vert \left\vert \left( Y_{s,t},\rho
_{2}\right) \right\vert +\left\vert \left( \widetilde{Y}_{s,t},\rho
_{1}\right) \right\vert \left\vert \left( Y_{s,t},\rho _{2}\right) -\left(
\widetilde{Y}_{s,t},\rho _{2}\right) \right\vert \\
&&+\sum_{\left\vert \lambda _{1}\right\vert +\left\vert \lambda
_{2}\right\vert \geq \left[ p\right] +1,\lambda _{i}\in F_{d}^{\left[ p%
\right] }}\left\vert B_{\rho _{1}}^{\lambda _{1}}B_{\rho _{2}}^{\lambda
_{2}}\left( X_{s,t},\lambda _{1}\right) \left( X_{s,t},\lambda _{2}\right)
\right\vert \\
&\leq &C_{p,\gamma ,d,\omega \left( 0,T\right) }\left\Vert f\right\Vert
_{Lip\left( \gamma -1\right) }^{\left\vert \rho \right\vert }\omega \left(
s,t\right) ^{\frac{\gamma }{p}}
\end{eqnarray*}

Suppose $\left\vert \tau \right\vert =k+1$ and $\tau \in F_{e}^{\left[ p%
\right] }\ $is a tree. Fix the starting point $s\in \lbrack 0,T)$. For $\rho
\in F_{e}^{\left[ p\right] }$, let $\beta _{s}^{\rho }$ denote the
slowly-varying one-form associated with the path $t\mapsto \left(
Y_{s,t},\rho \right) $ for $t\geq s$. Combining estimates of one-forms in
Lemma \ref{Lemma effects are stable under multiplication} and Lemma \ref%
{Lemma effects are stable under iterated integral}, we have
\begin{equation}
\left\Vert \beta _{s}^{\rho }\right\Vert _{\frac{\gamma }{p}}\leq
C_{p,\gamma ,d,\omega \left( 0,T\right) }\left\Vert f\right\Vert _{Lip\left(
\gamma -1\right) }^{\left\vert \rho \right\vert }.
\label{inner estimate beta s rho}
\end{equation}%
Based on the formulae for one-forms in Lemma \ref{Lemma effects are stable
under multiplication} and Lemma \ref{Lemma effects are stable under iterated
integral}, the part of $\beta _{s}^{\rho }\left( X_{s}\right) $ that
consists of integrals of one-forms on $\left[ 0,s\right] $ vanishes (since
the starting point here is $s$, not $0$), leaving only the diagonal term.
Hence, it can be proved by induction on $\left\vert \rho \right\vert $ for $%
\rho \in F_{e}^{\left[ p\right] }$ that $\beta _{s}^{\rho }\left(
X_{s}\right) =B_{\rho }$ in Definition \ref{Definition of Y tilde}. Then
based on Definition \ref{Definition of Y tilde}, for the tree $\tau $ that
we are considering, $(\widetilde{Y}_{s,t},\tau ):=B_{\tau }\left(
X_{s,t}\right) =\beta _{s}^{\tau }\left( X_{s}\right) \left( X_{s,t}\right) $%
. On the other hand, based on Definition \ref{Definition of rough integral},
$\left( Y_{s,t},\tau \right) =\int_{r=s}^{t}\beta _{s}^{\tau }\left(
X_{r}\right) dX_{r}$. By combining $\left( \ref{inner estimate beta s rho}%
\right) $ for the tree $\tau $ with Lemma \ref{Lemma estimate of increments
of effects}, we have
\begin{eqnarray*}
\left\vert \left( Y_{s,t},\tau \right) -\left( \widetilde{Y}_{s,t},\tau
\right) \right\vert &=&\left\vert \int_{r=s}^{t}\beta _{s}^{\tau }\left(
X_{r}\right) dX_{r}-\beta _{s}^{\tau }\left( X_{s}\right) \left(
X_{s,t}\right) \right\vert \\
&\leq &C_{p,\gamma ,d,\omega \left( 0,T\right) }\left\Vert f\right\Vert
_{Lip\left( \gamma -1\right) }^{\left\vert \tau \right\vert }\omega \left(
s,t\right) ^{\frac{\gamma }{p}}.
\end{eqnarray*}
\end{proof}

\begin{proof}[Proof of Theorem \protect\ref{Theorem integral of branched
rough paths}]
Based on Definition \ref{Definition of Y tilde}, for each $\rho \in F_{e}^{%
\left[ p\right] }$, $(\widetilde{Y}_{s,t},\rho )$ is a linear combination of
$\left( X_{s,t},\lambda \right) $ for $\lambda \in F_{d}^{\left[ p\right] }$%
, $\left\vert \lambda \right\vert \geq \left\vert \rho \right\vert $, with
the coefficients given by linear combinations of derivatives of $f$ at $\pi
_{1}\left( X_{s}\right) $ of order $0,1,\dots ,\left[ p\right] -1$. Since $f$
is $Lip\left( \gamma -1\right) $ for $\gamma >p$, these coefficients are at
least $\left( \gamma -\left[ p\right] \right) $-H\"{o}lder continuous.
Suppose $X^{n},n=1,2,\dots $ is a sequence of branched $p$-rough paths that
converges to $X$ in $d_{p}$-metric. Based on Definition \ref{Definition of Y
tilde}, define $\widetilde{Y}^{n}$ and $\widetilde{Y}$ associated with $%
X^{n} $ and $X$ respectively. Based on Minkowski inequality and the
monotonicity of $l_{p}$ norm, $\widetilde{Y}^{n}$ converge to $\widetilde{Y}$
in $d_{p}$-metric as $n\rightarrow \infty $. Based on $\left( \ref{estimate
Y s,t rho and Y tilde s,t rho are close}\right) $ and a uniqueness proof
that is similar to that of \cite[Theorem 4.3]{lyons2007differential}, $Y$ is
the unique branched $p$-rough path associated with $\tilde{Y}$. Then based
on an argument similar to \cite[Theorem 3.2.2]{lyons2002system}, $Y$ is
continuous with respect to $X$ in $d_{p}$-metric.

The constant $C_{p,\gamma ,d,\omega \left( 0,T\right) }$ in $\left( \ref%
{estimate norm on Ys,t rho}\right) $ can be chosen to be monotone increasing
in $\omega \left( 0,T\right) $. Hence, when $\omega \left( s,t\right) \leq 1$%
,%
\begin{equation}
\left\Vert Y_{s,t}\right\Vert :=\max_{\rho \in F_{e}^{\left[ p\right]
}}\left\vert \left( Y_{s,t},\rho \right) \right\vert ^{\frac{1}{\left\vert
\rho \right\vert }}\leq C_{p,\gamma ,d}\left\Vert f\right\Vert _{Lip\left(
\gamma -1\right) }\omega \left( s,t\right) ^{\frac{1}{p}}
\label{inner estimate Ys,t norm}
\end{equation}%
Let $\mathcal{G}_{e}^{\left[ p\right] }$ denote the step-$\left[ p\right] $
truncated group of grouplike elements in Grossman Larson Hopf algebra
labelled with $\left\{ 1,\dots ,e\right\} $. Define $\left\Vert \cdot
\right\Vert ^{\prime }$ on $\mathcal{G}_{e}^{\left[ p\right] }$ as in \cite[%
Definition 3.2]{yang2022remainder}. Then $\left\Vert \cdot \right\Vert
^{\prime }$ satisfies the triangle inequality, and is a continuous
homogeneous norm \cite{yang2022remainder}. For $a\in G_{e}^{\left[ p\right]
} $, define $\overline{a}:\left\{ \rho |\rho \in F_{e},\left\vert \rho
\right\vert \leq \left[ p\right] \right\} \rightarrow
\mathbb{R}
$ as $\left( \overline{a},\epsilon \right) =1$ and $\left( \overline{a},\rho
\right) =\left( a,\rho \right) /\sigma \left( \rho \right) $ for $\rho \in
F_{e}^{\left[ p\right] }$ with $\sigma \left( \rho \right) $ at $\left( \ref%
{definition of symmetry factor}\right) $. Then based on \cite[Theorem 43]%
{foissy2002algebresb}\cite[Proposition 4.4]{hoffman2003combinatorics}, $%
\overline{a}\in \mathcal{G}_{e}^{\left[ p\right] }$. Since all continuous
homogeneous norms on $\mathcal{G}_{e}^{\left[ p\right] }$ are equivalent
\cite[Proposition 3.4]{yang2022remainder}, $\left\Vert \overline{a}%
\right\Vert ^{\prime }$ is equivalent to $\max_{\rho \in F_{e}^{\left[ p%
\right] }}\left\vert \left( \overline{a},\rho \right) \right\vert ^{\frac{1}{%
\left\vert \rho \right\vert }}$ up to a constant depending on $p,e$. Then
with $\left\Vert a\right\Vert :=\max_{\rho \in F_{e}^{\left[ p\right]
}}\left\vert \left( a,\rho \right) \right\vert ^{\frac{1}{\left\vert \rho
\right\vert }}$ for $a\in G_{e}^{\left[ p\right] }$, there exist $%
0<c_{p,e}<C_{p,e}<\infty $ such that for every $a\in G_{e}^{\left[ p\right]
} $, $c_{p,e}\left\Vert \overline{a}\right\Vert ^{\prime }\leq \left\Vert
a\right\Vert \leq C_{p,e}\left\Vert \overline{a}\right\Vert ^{\prime }$.
Since $\left\Vert \cdot \right\Vert ^{\prime }$ satisfies the triangle
inequality, based on $\left( \ref{inner estimate Ys,t norm}\right) $ and
\cite[Proposition 5.10]{friz2010multidimensional},
\begin{equation*}
\left\Vert Y\right\Vert _{p-var,\left[ s,t\right] }\leq C_{p,\gamma
,d,e}\left\Vert f\right\Vert _{Lip\left( \gamma -1\right) }\left( \omega
\left( s,t\right) ^{\frac{1}{p}}\vee \omega \left( s,t\right) \right)
\end{equation*}
\end{proof}

\begin{proof}[Proof of Proposition \protect\ref{Proposition same first level
integrals}]
Fix $s\in \left[ 0,T\right) $, denote $x_{s}:=\pi _{1}\left( X_{s}\right) $.
Define a polynomial one-form $p:%
\mathbb{R}
^{d}\rightarrow L\left(
\mathbb{R}
^{d},%
\mathbb{R}
^{e}\right) $ as%
\begin{equation*}
p\left( u\right) :=\sum_{k=0}^{\left[ p\right] -1}\frac{\left( D^{k}f\right)
\left( x_{s}\right) }{k!}\left( u-x_{s}\right) ^{\otimes k}\text{ for }u\in
\mathbb{R}
^{d}\text{.}
\end{equation*}%
Then
\begin{equation}
\left( D^{k}p\right) \left( x_{s}\right) =\left( D^{k}f\right) \left(
x_{s}\right) \text{ for }k=0,1,\dots ,\left[ p\right] -1.
\label{inner estimate derivatives of p equals f}
\end{equation}%
Based on Whitney Extension Theorem \cite[p.177, Theorem 4]{stein1970singular}%
, we restrict $p$ to $\left\{ u\in
\mathbb{R}
^{d}|\left\vert u-x_{s}\right\vert \leq 1\right\} $ and extend $p$ globally
to a $Lip\left( \gamma \right) $ one-form on $%
\mathbb{R}
^{d}$ for $\gamma >p$. We still denote the extension by $p$.

Define $F:%
\mathbb{R}
^{d+e}\rightarrow L\left(
\mathbb{R}
^{d},%
\mathbb{R}
^{d+e}\right) $ as $F\left( x,y\right) \left( v\right) =\left( v,p\left(
x\right) \left( v\right) \right) $ for $\left( x,y\right) \in
\mathbb{R}
^{d+e}$ and $v\in
\mathbb{R}
^{d}$. Let $w$ be the solution of the branched rough differential equation%
\begin{equation*}
dw_{t}=F\left( w_{t}\right) dX_{t},w_{s}=\left( x_{s},0\right) .
\end{equation*}%
Since $F$ is $Lip\left( \gamma \right) $ for $\gamma >p$, the solution $w$
exists uniquely \cite[Theorem 22]{lyons2015theory}.

Suppose $F=\left( F_{1},\cdots ,F_{d}\right) $ with $F_{i}:%
\mathbb{R}
^{d+e}\rightarrow
\mathbb{R}
^{d+e}$. For $\tau _{j}\in T_{d}^{\left[ p\right] },j=1,\dots ,l$ and $%
i=1,\dots ,d$, define recursively $F\left( \left[ \tau _{1}\cdots \tau _{l}%
\right] _{i}\right) :=\left( D^{l}F_{i}\right) \left( F\left( \tau
_{1}\right) \cdots F\left( \tau _{l}\right) \right) $ with $F\left( \bullet
_{i}\right) =F_{i}$ for $i=1,\dots ,d$. Recall $\nu _{k}$, $k=1,\dots ,K$ in
Notation \ref{Notation Bd[p]}. Assume $\nu _{i}=\bullet _{i}$ for $i=1,\dots
,d$. Define one-form $\widetilde{F}:%
\mathbb{R}
^{d+e}\rightarrow L\left(
\mathbb{R}
^{K},%
\mathbb{R}
^{d+e}\right) $ as $\widetilde{F}:=\left( F\left( \nu _{1}\right) ,\cdots
,F\left( \nu _{K}\right) \right) $. Let $\widetilde{w}$ be the solution of%
\begin{equation*}
d\widetilde{w}_{t}=\widetilde{F}\left( \widetilde{w}_{t}\right) dZ_{t},%
\widetilde{w}_{s}=\left( x_{s},0\right)
\end{equation*}%
which can be spelt out as%
\begin{eqnarray*}
dx_{t}^{i} &=&dZ_{t}^{i}\text{ for }i=1,\dots ,d \\
dy_{t} &=&\sum_{k=1}^{K}p_{\nu _{k}}\left( x_{t}\right) dZ_{t}^{k} \\
\left( x_{s},y_{s}\right) &=&\left( x_{s},0\right) ,
\end{eqnarray*}%
where $p_{\nu _{k}},k=1,\dots ,K$ are defined in Notation \ref{Notation f
tau in integrals}.

Since $F\ $is $Lip\left( \gamma \right) $ for $\gamma >p$, based on \cite[%
Proposition 3.16]{yang2022remainder}, $w$ and $\widetilde{w}$ have the same
step-$\left[ p\right] $ Taylor expansion on $\left[ s,t\right] $. Then
combined with $\left( \ref{inner estimate derivatives of p equals f}\right) $%
,
\begin{eqnarray}
&&\sum_{\tau \in T_{d}^{\left[ p\right] }}f_{\tau }\left( x_{s}\right) \frac{%
\left( X_{s,t},\tau \right) }{\sigma \left( \tau \right) }
\label{inner same first level taylor expansion} \\
&=&\sum_{k=1}^{K}\sum_{\substack{ i_{1},\cdots ,i_{l}\in \left\{ 1,\dots
,d\right\}  \\ l=0,\cdots ,\left[ p\right] -\left\vert \nu _{k}\right\vert }}%
D_{i_{1},\cdots ,i_{l}}^{l}f_{\nu _{k}}\left( x_{s}\right) \left(
Z_{s,t},i_{1}\cdots i_{l}k\right)  \notag
\end{eqnarray}%
Based on Lemma \ref{Lemma beta is a slowly varying one-form}, the limit of
the Riemann sums of the left hand side of $\left( \ref{inner same first
level taylor expansion}\right) $ is the first level of the branched rough
integral $\int f\left( X\right) dX$. On the other hand, denote $\Gamma
=\left( \gamma _{1},\cdots ,\gamma _{K}\right) $ with $\gamma _{k}:=\left(
\gamma -\left\vert \nu _{k}\right\vert \right) /p$ and denote $\Pi =\left(
p_{1},\cdots ,p_{K}\right) $ with $p_{k}=p/\left\vert \nu _{k}\right\vert $.
Based on \cite[Definition 3.2]{gyurko2016differential}, $\left\{ f_{\nu
_{k}}\right\} _{k=1}^{K}$ is a $Lip^{\Gamma ,\Pi }$ one-form. Since $\gamma
_{k}>1-1/p_{k}$, $k=1,\dots ,K$, based on \cite[Definition 3.10]%
{gyurko2016differential}, the limit of the Riemann sums of the right hand
side of $\left( \ref{inner same first level taylor expansion}\right) $ is
the first level of the integral $\int g\left( z\right) dZ$ with $g=\left(
g_{1},\cdots ,g_{K}\right) :%
\mathbb{R}
^{d}\rightarrow L\left(
\mathbb{R}
^{K},%
\mathbb{R}
^{e}\right) $ given by $g_{k}=f_{\nu _{k}}$, $k=1,\dots ,K$.
\end{proof}

\bibliographystyle{unsrt}
\bibliography{acompat,roughpath}

\end{document}